\pdfoutput=1
\documentclass[oneside,english]{amsart}
\usepackage[T1]{fontenc}
\usepackage[latin9]{inputenc}
\usepackage{amsthm}
\usepackage{amstext}
\usepackage{amssymb}

\makeatletter

\newcommand{\noun}[1]{\textsc{#1}}

\numberwithin{equation}{section} 
\numberwithin{figure}{section} 
\theoremstyle{plain}
\theoremstyle{plain}
\newtheorem{thm}{Theorem}
  \theoremstyle{definition}
  \newtheorem{defn}[thm]{Definition}
 \theoremstyle{definition}
  \newtheorem{example}[thm]{Example}
  \theoremstyle{plain}
  \newtheorem{lem}[thm]{Lemma}

\usepackage[OT1]{fontenc} 

\makeatother

\usepackage{babel}

\begin{document}

\title{On Element SDD Approximability}

\author{Haim Avron, Gil Shklarski, and Sivan Toledo}

\date{20 May 2008}
\begin{abstract}
This short communication shows that in some cases scalar elliptic
finite element matrices cannot be approximated well by an \noun{sdd}
matrix. We also give a theoretical analysis of a simple heuristic
method for approximating an element by an \noun{sdd} matrix.
\end{abstract}
\maketitle

\section{Introduction}

In~\cite{ACST09} we study the theoretical and practical aspects
of approximation scalar elliptic finite element matrices by \noun{sdd}
matrices. The core task in the process is element-by-element approximation
of reasonably small element matrices. This short communication shows
that in some cases element matrices cannot be approximated well by
an \noun{sdd} matrix. We also give a theoretical analysis of a simple
heuristic method for approximating an element by an \noun{sdd} matrix.

There are several different definition of condition numbers and generalized
condition number. We give here the definition that is suitable for
analyzing preconditioners, which is the ultimate goal of our approximation.
\begin{defn}
Given two real symmetric positive semidefinite matrices $A$ and $B$
with the same null space $\mathbb{S}$, a \emph{finite generalized
eigenvalue} $\lambda$ of $(A,B)$ is a scalar satisfying $Ax=\lambda Bx$
for some $x\not\in\mathbb{S}$. The \emph{generalized finite spectrum}
$\Lambda(A,B)$ is the set of finite generalized eigenvalues of $(A,B)$,
and the \emph{generalized condition number} $\kappa(A,B)$ is \[
\kappa(A,B)=\frac{\max\Lambda(A,B)}{\min\Lambda(A,B)}\;.\]

\end{defn}
We also define the condition number $\kappa(A)=\kappa(A,P_{\perp\mathbb{S}})$
of a single matrix $A$ with null-space $\mathbb{S}$, where $P_{\perp\mathbb{S}}$
is the orthogonal projector onto the subspace orthogonal to $\mathbb{S}$.
This scalar is the ratio between the maximal eigenvalue and the minimal
nonzero eigenvalue of $A$.

\section{Approximability of Ill-Conditioned Matrices}

\subsection{Approximability vs. Ill-conditioning}

In \cite{ACST09} we show that if $A$ is well-conditioned it is always
well-approximable. The question is whether this sufficient condition
is also a necessary one. When $A$ is ill conditioned, there may or
may not be an \noun{sdd} matrix $B$ that approximates it well; the
following two examples demonstrate both cases.

Example~2.13 in~\cite{ACST09} presents an \noun{sdd} matrix that
is ill conditioned. This is an obvious example of an ill-conditioned
but well-approximable matrix. First, we show an example of a non-\noun{sdd}
(and not close to \noun{sdd}) ill-conditioned matrix which is still
well-approximable. We 
\begin{example}
(\cite{Avron10}) Let\[
A=\frac{1}{6\epsilon}\left[\begin{array}{cccccc}
3(1+\epsilon^{2}) & \epsilon^{2} & 1 & -4\epsilon^{2} & 0 & -4\\
\epsilon^{2} & 3\epsilon^{2} & 0 & -4\epsilon^{2} & 0 & 0\\
1 & 0 & 3 & 0 & 0 & -4\\
-4\epsilon^{2} & -4\epsilon^{2} & 0 & 8(1+\epsilon^{2}) & -8 & 0\\
0 & 0 & 0 & -8 & 8(1+\epsilon^{2}) & -8\epsilon^{2}\\
-4 & 0 & -4 & 0 & -8\epsilon^{2} & 8(1+\epsilon^{2})\end{array}\right]\]
for some small $\epsilon>0$. This matrix is the element matrix for
a quadratic triangular element with nodes $(0,0)$, $(0,\epsilon)$
and $(1,0)$, quadrature points are midpoints of the edges with equal
weights, and material constant $\theta=1$.
\end{example}
This matrix is clearly ill conditioned since the maximum ratio between
its diagonal elements is proportional to $1/\epsilon^{2}$. 

To show that this matrix is approximable consider the following \noun{sdd}
matrix: \[
A_{+}=\frac{1}{6\epsilon}\left[\begin{array}{cccccc}
4(1+\epsilon^{2}) & 0 & 0 & -4\epsilon^{2} & 0 & -4\\
0 & 4\epsilon^{2} & 0 & -4\epsilon^{2} & 0 & 0\\
0 & 0 & 4 & 0 & 0 & -4\\
-4\epsilon^{2} & -4\epsilon^{2} & 0 & 8(1+\epsilon^{2}) & -8 & 0\\
0 & 0 & 0 & -8 & 8(1+\epsilon^{2}) & -8\epsilon^{2}\\
-4 & 0 & -4 & 0 & -8\epsilon^{2} & 8(1+\epsilon^{2})\end{array}\right]\,.\]
 We will show that $\kappa(A,A_{+})\leq2$. Define the matrix \[
A_{-}=\frac{1}{6\epsilon}\left[\begin{array}{cccccc}
(1+\epsilon^{2}) & -\epsilon^{2} & -1 & 0 & 0 & 0\\
-\epsilon^{2} & \epsilon^{2} & 0 & 0 & 0 & 0\\
-1 & 0 & 1 & 0 & 0 & 0\\
0 & 0 & 0 & 0 & 0 & 0\\
0 & 0 & 0 & 0 & 0 & 0\\
0 & 0 & 0 & 0 & 0 & 0\end{array}\right]\,.\]
Notice that $A=A_{+}-A_{-}$. Since $A_{-}$ is symmetric positive
definite this implies that $\lambda_{\max}(A,A_{+})\leq1$. We will
show that $\lambda_{\min}(A,A_{+})\geq1/2$. This condition is equivalent
to the condition that $\lambda_{\max}(A_{+},A)\leq2$, which is the
equivalent to the condition that $2A-A_{+}$ is positive semidefinite.
According to Lemma 3.3 in \cite{BGHNT} it is enough to prove that
$\bar{\sigma}(A_{-},A_{+})\leq\frac{1}{2}$. This can easily be achieved
by path embedding where we embed the $(1,2)$ edge in $A_{-}$ with
the path$(1,4)\rightarrow(4,2)$ and the edge $(1,3)$ with the path
$(1,6)\rightarrow(6,3)$. Congestion is $1$ because no edge is reused
and for both paths the dilation is $\frac{1}{2}$.

\subsection{Pathological Inapproximability}

The question that we answer in this subsection is whether there exist
inapproximable elements. By the last section, if there exist such
an inapproximable element matrix, it should be ill conditioned. The
following example show that such a matrix indeed exist.
\begin{example}
(\cite{Shklarski08}) The following matrix is an element matrix for
an isosceles triangle with two tiny angles and one that is almost
$\pi$, with nodes at $(0,0)$, $(1,0)$, and $(1/2,\epsilon)$ for
some small $\epsilon>0$. The element matrix is \[
A=\frac{1}{2\epsilon}\begin{bmatrix}\frac{1}{4}+\epsilon^{2} & \frac{1}{4}-\epsilon^{2} & -\frac{1}{2}\\
\\\frac{1}{4}-\epsilon^{2} & \frac{1}{4}+\epsilon^{2} & -\frac{1}{2}\\
\\-\frac{1}{2} & -\frac{1}{2} & 1\end{bmatrix}\]
This matrix has rank $2$ and null vector $\begin{bmatrix}1 & 1 & 1\end{bmatrix}^{T}$.
We now show that for any \noun{sdd} matrix $B$ with the same null
space, $\kappa(A,B)\geq\epsilon^{-2}/4$. Let $v=\begin{bmatrix}1 & -1 & 0\end{bmatrix}^{T}$
and $u=\begin{bmatrix}1 & 1 & -2\end{bmatrix}^{T}$; both are orthogonal
to $\begin{bmatrix}1 & 1 & 1\end{bmatrix}^{T}$. We have $v^{T}Av=2\epsilon$
and $u^{T}Au=4.5\epsilon^{-1}$. Therefore,\begin{eqnarray*}
\kappa(A,B) & = & \max_{x\perp\text{null}(A)}\frac{x^{T}Ax}{x^{T}Bx}\times\max_{x\perp\text{null}(A)}\frac{x^{T}Bx}{x^{T}Ax}\\
 & \geq & \frac{u^{T}Au}{u^{T}Bu}\times\frac{v^{T}Bv}{v^{T}Av}\\
 & = & \frac{4.5}{2\epsilon^{2}}\times\frac{v^{T}Bv}{u^{T}Bu}\;.\end{eqnarray*}

We denote the entries of $B$ by \[
B=\begin{bmatrix}b_{12}+b_{13} & -b_{12} & -b_{13}\\
-b_{12} & b_{12}+b_{23} & -b_{23}\\
-b_{13} & -b_{23} & b_{13}+b_{23}\end{bmatrix}\]
where the $b_{ij}$'s are non-negative. Furthermore, at least two
of the $b_{ij}$'s must be positive, otherwise $B$ will have rank
$1$ or $0$, not rank $2$. In particular, $b_{13}+b_{23}>0$. This
gives\[
\frac{v^{T}Bv}{u^{T}Bu}=\frac{4b_{12}+b_{13}+b_{23}}{9b_{13}+9b_{23}}=\frac{4b_{12}}{9b_{13}+9b_{23}}+\frac{1}{9}\geq\frac{1}{9}\;.\]
Therefore, $\kappa(A,B)>\epsilon^{-2}/4$, which can be arbitrarily
large.
\end{example}

\section{\label{subsec:positive-part}A Simple Heuristic for Symmetric Diagonally-Dominant
Approximations}

The following definition presents a heuristic for \noun{sdd} approximation. 
\begin{defn}
(\cite{ACST09}) Let $A$ be a symmetric positive (semi)definite matrix.
We define $A_{+}$ to be the \noun{sdd} matrix defined by\[
\left(A_{+}\right)_{ij}=\begin{cases}
a_{ij} & i\neq j\mbox{ and }a_{ij}<0\\
0 & i\neq j\mbox{ and }a_{ij}\geq0\\
\sum_{k\neq j}-\left(A_{+}\right)_{ik} & i=j\;.\end{cases}\]
Clearly, $A_{+}$ is \noun{sdd}. We show that if $A$ is well conditioned
this simple heuristic yields a fairly good approximation.\end{defn}
\begin{lem}
(\cite{Shklarski08}) Let $A$ be an \noun{spsd }$n_{e}$-by-$n_{e}$
matrix with $\textrm{null}(A)=\textrm{span}[1\ldots1]^{T}$. Then
$\textrm{null}(A_{+})=\textrm{span}[1\ldots1]^{T}$, and $\kappa(A,A_{+})\leq\sqrt{n_{e}}\kappa(A)$.
Moreover, if there exist a constant $c$ and an index $i$ such that
$\left\Vert A\right\Vert _{\inf}\leq cA_{ii}$ , then $\kappa(A,A_{+})\leq c\kappa(A)$.\end{lem}
\begin{proof}
We first show that $\textrm{null}(A)=\textrm{null}(A_{+})$. Let $A_{-}=A_{+}-A$.
The matrix $A_{-}$ is symmetric and contains only nonpositive off
diagonals, and \begin{eqnarray*}
A_{-}\begin{bmatrix}1 & \ldots & 1\end{bmatrix}^{T} & = & A_{+}\begin{bmatrix}1 & \ldots & 1\end{bmatrix}^{T}-A\begin{bmatrix}1 & \ldots & 1\end{bmatrix}^{T}=0\;.\end{eqnarray*}
Therefore, $A_{-}$ is an \noun{sdd} matrix. Since \noun{sdd} matrices
are also \noun{spsd}, for all $x$,

\begin{equation}
0\leq x^{T}Ax=x^{T}A_{+}x-x^{T}A_{-}x\leq x^{T}A_{+}x\;.\label{eq:xAx_is_less_xAplusx}\end{equation}
Therefore, $\mbox{null}(A_{+})\subseteq\textrm{null}(A)$. The equality
of these linear spaces follows from the equation $A_{+}\begin{bmatrix}1 & \ldots & 1\end{bmatrix}^{T}=0$. 

By equation~\ref{eq:xAx_is_less_xAplusx}, for all $x\notin\mbox{null}(A_{+})$,
$x^{T}Ax/x^{T}A_{+}x\leq1$. This shows that $\max\Lambda(A,A_{+})\leq1$. 

We now bound $\max\Lambda(A_{+},A)$ from above. For every vector
$x\notin\mbox{null}(A)$, \[
\frac{x^{T}A_{+}x}{x^{T}Ax}=\frac{x^{T}A_{+}x/x^{T}x}{x^{T}Ax/x^{T}x}\leq\frac{\max\Lambda(A_{+})}{\min\Lambda(A)}\;.\]
Therefore, it is is sufficient to show that $\max\Lambda(A_{+})\leq\theta\max\Lambda(A)$
for some positive $\theta$ in order to show that $\max\Lambda(A_{+},A)\leq\theta\kappa(A)$
and $\kappa(A,A_{+})\leq\theta\kappa(A)$. 

Since $A\begin{bmatrix}1 & \ldots & 1\end{bmatrix}^{T}=0$, $A$ is
\noun{spsd}, and assuming $n_{e}>1$, for every $i$,\[
\left|A_{ii}\right|=A_{ii}=\sum_{j\neq i}\left|\left(A_{+}\right)_{ij}\right|-\sum_{j\neq i}\left|\left(A_{-}\right)_{ij}\right|\;.\]
Therefore, since $A_{+}\begin{bmatrix}1 & \ldots & 1\end{bmatrix}^{T}=0$,
for every $i$,

\begin{eqnarray}
\sum_{j}\left|A_{ij}\right| & = & \left|A_{ii}\right|+\sum_{j\neq i}\left|\left(A_{+}\right)_{ij}\right|+\sum_{j\neq i}\left|\left(A_{-}\right)_{ij}\right|\nonumber \\
 & = & 2\sum_{j\neq i}\left|\left(A_{+}\right)_{ij}\right|\nonumber \\
 & = & \left(A_{+}\right)_{ii}+\sum_{j\neq i}\left|\left(A_{+}\right)_{ij}\right|\nonumber \\
 & = & \sum_{j}\left|\left(A_{+}\right)_{ij}\right|\;.\label{eq:row sums A, A_{+}}\end{eqnarray}
By the definitions of the $1$-norm and the $\inf$-norm, and the
fact that $A_{+}$ is symmetric, $\left\Vert A_{+}\right\Vert _{1}=\left\Vert A_{+}\right\Vert _{\inf}$.
Moreover, by \cite[Corollary 2.3.2]{GolubVa89}, $\left\Vert A_{+}\right\Vert _{2}\leq\sqrt{\left\Vert A_{+}\right\Vert _{1}\left\Vert A_{+}\right\Vert _{\inf}}$.
Therefore,

\begin{eqnarray}
\left\Vert A_{+}\right\Vert _{2} & \leq & \left\Vert A_{+}\right\Vert _{\inf}\nonumber \\
 & = & \max_{i}\sum_{j}\left|\left(A_{+}\right)_{ij}\right|\nonumber \\
 & = & \max_{i}\sum_{j}\left|A_{ij}\right|\label{eq:|A+|2<|A|inf}\\
 & = & \left\Vert A\right\Vert _{\inf}\;,\nonumber \end{eqnarray}
where the second equality is due to equation~\ref{eq:row sums A, A_{+}}.
Therefore, by~\cite[Equation 2.3.11]{GolubVa89} \[
\left\Vert A_{+}\right\Vert _{2}\leq\left\Vert A\right\Vert _{\inf}\leq\sqrt{n_{e}}\left\Vert A\right\Vert _{2}\;.\]
Since $A_{+}$ and $A$ are both symmetric, $\max\Lambda(A_{+})=\left\Vert A_{+}\right\Vert _{2}$
and $\max\Lambda(A)=\left\Vert A\right\Vert _{2}$. Therefore, $\max\Lambda(A_{+})\leq\sqrt{n_{e}}\max\Lambda(A)$.
This shows that $\kappa(A,A_{+})\leq\sqrt{n_{e}}\kappa(A)$ and concludes
the proof of the first part of the lemma.

We now assume that there exist a constant $c$ and an index $i$,
such that $\left\Vert A\right\Vert _{\inf}\leq cA_{ii}$. By equation~\ref{eq:|A+|2<|A|inf},
we have that $\left\Vert A_{+}\right\Vert _{2}\leq cA_{ii}$. Since
for every $i$, $A_{ii}\leq\left\Vert A\right\Vert _{2}$, we have
that $\left\Vert A_{+}\right\Vert _{2}\leq c\left\Vert A\right\Vert _{2}$.
Therefore, $\max\Lambda(A_{+})\leq c\max\Lambda(A)$. This shows that
in this case $\kappa(A,A_{+})\leq c\kappa(A)$ and concludes the proof
of the lemma.

\end{proof}
The following example shows that if $A$ is not well-conditioned,
this heuristic may generate a bad approximation.
\begin{example}
(\cite{Avron10}) Let $0<\epsilon\ll1$, and let $M\geq\frac{4}{\epsilon}$,
\[
A=\begin{bmatrix}1+M & -1 & 0 & -M\\
-1 & 1+M & -M & 0\\
0 & -M & M & 0\\
-M & 0 & 0 & M\end{bmatrix}-\begin{bmatrix}0 & 0 & 0 & 0\\
0 & 0 & 0 & 0\\
0 & 0 & 1-\epsilon & -1+\epsilon\\
0 & 0 & -1+\epsilon & 1-\epsilon\end{bmatrix}\;.\]
This matrix is symmetric semidefinite with rank $3$ and null vector
$\begin{bmatrix}1 & 1 & 1 & 1\end{bmatrix}^{T}$. We show that for
small $\epsilon$, $A$ is ill conditioned, with condition number
larger than $8\epsilon^{-2}$. Let\[
q_{1}=\frac{1}{2}\begin{bmatrix}1\\
1\\
1\\
1\end{bmatrix}\;,\quad q_{2}=\frac{1}{2}\begin{bmatrix}1\\
1\\
-1\\
-1\end{bmatrix}\;,\quad q_{3}=\frac{1}{2}\begin{bmatrix}1\\
-1\\
1\\
-1\end{bmatrix}\;,\;\mbox{and}\quad q_{4}=\frac{1}{2}\begin{bmatrix}1\\
-1\\
-1\\
1\end{bmatrix}\]
be an orthonormal basis for $\mathbb{R}^{4}$. We have\begin{eqnarray*}
q_{1}^{T}Aq_{1} & = & 0\\
q_{2}^{T}Aq_{2} & = & 2M\\
q_{3}^{T}Aq_{3} & = & 2M+\epsilon\\
q_{4}^{T}Aq_{4} & = & \epsilon\;.\end{eqnarray*}
Therefore, $\kappa(A)\geq2M/\epsilon\geq8\epsilon^{-2}$. We show
that the matrix $A_{+}$ is a poor approximation of $A$.\begin{eqnarray*}
q_{1}^{T}A_{+}q_{1} & = & 0\\
q_{2}^{T}A_{+}q_{2} & = & 2M\\
q_{3}^{T}A_{+}q_{3} & = & 2M+1\\
q_{4}^{T}A_{+}q_{4} & = & 1\;.\end{eqnarray*}
Therefore, \[
\kappa(A,A_{+})>\left(1-\frac{1-\epsilon}{2M+1}\right)\epsilon^{-1}\approx\epsilon^{-1}\;.\]
On the other hand, the \noun{sdd} matrix\[
B=\begin{bmatrix}\epsilon+M & -\epsilon & 0 & -M\\
-\epsilon & \epsilon+M & -M & 0\\
0 & -M & \epsilon+M & -\epsilon\\
-M & 0 & -\epsilon & \epsilon+M\end{bmatrix}\]
is a good approximation of $A$, with $\kappa(A,B)<9$. This bound
follows from a simple path-embedding arguments~\cite{BGHNT}, which
shows that $3A-B$ and $3B-A$ are positive semidefinite. The quantitative
parts of these arguments rest on the inequalities\[
\frac{1}{2M}+\frac{1}{2M}+\frac{1}{3-\epsilon}\leq\frac{1}{3-4\epsilon}\]
and\[
\frac{1}{2M}+\frac{1}{2M}+\frac{1}{1+2\epsilon}<\frac{1}{1-3\epsilon}\;,\]
which hold for small $\epsilon$.

\bibliographystyle{plain}
\bibliography{element_approximability}

\end{example}

\end{document}